\documentclass{amsart}
\usepackage{amssymb,
enumitem,
mathrsfs,
hyperref,
tikz,
upgreek,
mathtools,
verbatim,
pbox
}
\definecolor{navyblue}{RGB}{0, 0, 128}
\hypersetup{colorlinks=true, allcolors=navyblue}
\numberwithin{equation}{section}
\usepackage[T1]{fontenc}
\usepackage[color = white, linecolor = black, shadow]{todonotes}

\renewcommand{\L}{\mathcal{L}}

\newcommand{\NN}{\mathbb{N}}
\newcommand{\ZZ}{\mathbb{Z}}

\newcommand{\XX}{\mathbb{X}}

\newcommand{\set}[1]{\{ #1 \}}
\newcommand{\setm}[2]{\{ #1 \mid #2 \}}

\newcommand{\lh}{\operatorname{lh}}

\newcommand{\dom}{\operatorname{dom}}

\newcommand{\iso}{\cong}
\newcommand{\embeds}{\sqsubseteq}
\newcommand{\analytic}{\boldsymbol{\Sigma}_1^1}
\newcommand{\coanalytic}{\boldsymbol{\Pi}_1^1}
\newcommand{\Aut}[1]{\mathrm{Aut}(#1)}
\newcommand{\Stab}[1]{\mathrm{Stab}(#1)}
\newcommand{\Subg}[1]{\mathrm{Subg}(#1)}

\newenvironment{enumerate-(a)}{\begin{enumerate}[label={\upshape (\alph*)}, leftmargin=2pc]}{\end{enumerate}}
\newenvironment{enumerate-(a)-r}{\begin{enumerate}[label={\upshape (\alph*)}, leftmargin=2pc,resume]}{\end{enumerate}}
\newenvironment{enumerate-(a)-5}{\begin{enumerate}[label={\upshape (\alph*)}, leftmargin=2pc,start=5]}{\end{enumerate}}
\newenvironment{enumerate-(A)}{\begin{enumerate}[label={\upshape (\Alph*)}, leftmargin=2pc]}{\end{enumerate}}
\newenvironment{enumerate-(A)-r}{\begin{enumerate}[label={\upshape (\Alph*)}, leftmargin=2pc,resume]}{\end{enumerate}}
\newenvironment{enumerate-(i)}{\begin{enumerate}[label={\upshape (\roman*)}, leftmargin=2pc]}{\end{enumerate}}
\newenvironment{enumerate-(i)-r}{\begin{enumerate}[label={\upshape (\roman*)}, leftmargin=2pc,resume]}{\end{enumerate}}
\newenvironment{enumerate-(I)}{\begin{enumerate}[label={\upshape (\Roman*)}, leftmargin=2pc]}{\end{enumerate}}
\newenvironment{enumerate-(I)-r}{\begin{enumerate}[label={\upshape (\Roman*)}, leftmargin=2pc,resume]}{\end{enumerate}}
\newenvironment{enumerate-(1)}{\begin{enumerate}[label={\upshape (\arabic*)}, leftmargin=2pc]}{\end{enumerate}}
\newenvironment{enumerate-(1)-r}{\begin{enumerate}[label={\upshape (\arabic*)}, leftmargin=2pc,resume]}{\end{enumerate}}
\newenvironment{itemizenew}{\begin{itemize}[leftmargin=2pc]}{\end{itemize}}

\newtheorem{theorem}{Theorem}[section]
\newtheorem{lemma}[theorem]{Lemma}
\newtheorem{corollary}[theorem]{Corollary}
\newtheorem{proposition}[theorem]{Proposition}
\newtheorem{question}[theorem]{Question}

\theoremstyle{definition}
\newtheorem{definition}[theorem]{Definition}

\theoremstyle{remark}

\numberwithin{equation}{section}

\begin{document}

\title[Invariant universality for quandles and fields]{Invariant universality for quandles and fields}
\date{\today}
\author[A.D.~Brooke-Taylor]{Andrew D. Brooke-Taylor}
\author[F.~Calderoni]{Filippo Calderoni}
\author[S.K.~Miller]{Sheila K. Miller}

\address{School of Mathematics, University of Leeds, Leeds, LS2 9JT --- United Kingdom}\email{a.d.brooke-taylor@leeds.ac.uk}
\address{Dipartimento di matematica \guillemotleft{Giuseppe Peano}\guillemotright, Universit\`a di Torino, Via Carlo Alberto 10, 10121 Torino --- Italy}
\email{filippo.calderoni@unito.it}
\address{ Department of Mathematics, City University of New York,
New York City College of Technology, 
300 Jay Street
Brooklyn, NY 11201 ---
USA}
\email{smiller@citytech.cuny.edu}

\thanks{
The first author was supported during this research by 
EPSRC Early Career Fellowship EP/K035703/2, 
``Bringing set theory and algebraic topology together.''
This work was carried out while the second author was visiting Rutgers University, partially supported by the ``National Group for the Algebraic and Geometric Structures and their Applications'' (GNSAGA--INDAM). The second author would like to thank Simon Thomas for interesting discussions and pointing out \cite{FriKol}}

\begin{abstract} 
We show that the embeddability relations for countable quandles and for 
countable fields of any given characteristic other than 2 
are maximally complex in a strong sense: they are \emph{invariantly universal}.
This notion from the theory of Borel reducibility states that any
analytic quasi-order on a standard Borel space essentially appears as the 
restriction of the embeddability relation to an isomorphism-invariant Borel set.
As an intermediate step we show that the embeddability relation 
of countable quandles is a complete analytic quasi-order. 
\end{abstract}

\maketitle


\section{Introduction}

The comparison of different equivalence relations in terms of Borel reducibility
has proven to be an extremely fruitful area of research, with implications 
in diverse areas of mathematics, most notably in showing that various 
classification programmes are impossible to complete satisfactorily.
See, for example, \cite{Hjo} for an introduction to the area;
note however that all necessary preliminaries for this paper will be provided in 
Section~\ref{sec : preliminaries}.
The area was initiated by the pioneering papers of H.~Friedman and L.~Stanley and
of Harrington, Kechris and Louveau \cite{FriSta, HarKecLou}, with the
former paper in particular focused on the equivalence relation of isomorphism 
between countable structures.
Indeed the set of all structures of a given type with underlying set the natural
numbers may be endowed with the topology of a complete separable metric space,
and in this framework the results of descriptive set theory have been brought to
bear on questions about equivalence relations to great effect.

In the underlying descriptive set-theoretic machinery, there is nothing that
requires us to constrain investigation to equivalence relations, and
recently attention in this field has expanded to include quasi-orders
(reflexive and transitive 
binary relations),
beginning with the work of Louveau and Rosendal \cite{LouRos}.
A central example of a quasi-order is the embeddability relation between
countable structures of a given type.
This also fits with previous work in category theory studying the complexity
of different categories, as for example in \cite{PulTrn}. 
Indeed, there is a kind of 
``Church's thesis for real mathematics'' that states that,
assuming the objects in question are reasonably encoded as members of a 
standard Borel space, hands-on 
constructions will invariably be Borel. 
Thus, from the functors between categories
that demonstrate universality one can expect to derive Borel reductions
that respect embeddings.
For example, building on work of Prze\'zdziecki \cite{Prz14} in a category-theoretic context,
the second author~\cite{Cal18} has shown 
that, when \(\kappa\) is an uncountable cardinal satisfying certain assumptions, the embeddability relation between \(\kappa\)-sized graphs Borel reduces in a generalised sense suitable for $\kappa$ to embeddability between \(\kappa\)-sized torsion-free abelian groups.

Louveau and Rosendal \cite{LouRos}
showed that within the class of analytic quasi-orders 
(see Section~\ref{sec : preliminaries} for definitions) there
are quasi-orders that are maximal with respect to Borel reducibility ---
so called \emph{complete analytic quasi-orders}.
Louveau and Rosendal furnish a number of examples, 
including the embeddability relation between
graphs.  In fact, the restriction of the graph embeddability relation to 
connected acyclic graphs --- \emph{combinatorial trees} --- is already 
complete analytic, a fact that we will make use of below.
We prove in Section~\ref{sec : quandles} that the embeddability relation on
quandles is complete analytic.  We also observe in Section~\ref{sec : fields}
that an old result of Fried and Koll\'ar \cite{FriKol}, when expressed in these
terms, states that the embeddability relation of fields is complete analytic.

When restricting to subclasses of structures, it is reasonable to consider the
case when the subclass is closed under isomorphism.  Thus arises the notion
of \emph{invariant universality} (Definition~\ref{Definition : invariantly universal}), 
first introduced by Camerlo, Marcone and
Motto~Ros \cite{CamMarMot} building on fundamental observations of S.~Friedman and
Motto~Ros \cite{FriMot}.
Whilst invariant universality imposes significant requirements making it stronger
than complete analyticity, a general trend observed in 
\cite{CamMarMot,CalMot} is that in practice, whenever the relation of 
embeddability on some space of countable structures is a complete analytic 
quasi-order, it is moreover invariantly universal with respect to isomorphism.

In Section~\ref{sec : invuniversality} of this paper 
we give the formal definition of invariant universality, and recall a special
case of Theorem~4.2 of \cite{CamMarMot}, which will be our main tool for
proving invariant universality. 
In Section~\ref{sec : quandles}
we first show that the embedding relation on countable quandles is
a complete analytic quasi-order, and then use this fact to show that 
the relation is invariantly universal. 
We further observe that arguing similarly we obtain invariant universality of
the embedding relations of related classes of countable
structures such as kei as LD-monoids.
In Section~\ref{sec : fields} we turn to
the embedding relation on fields of a given characteristic other than 2.
In this case, the fact that the embeddability
relation is complete analytic was essentially shown by Fried and
Koll\'ar \cite{FriKol}, and arguing using their construction we are able to 
show that the relation is invariantly universal.
Our results all 
add weight to the trend mentioned above, and hint that in the search
for a natural example of a 
complete analytic quasi-order that is not invariantly universal,
it might be best to focus on relations other than embeddability.


\section{Preliminaries}\label{sec : preliminaries}

A \emph{standard Borel space} is a pair \( (X,\mathcal B) \) such that \( \mathcal B \) is the \( \sigma \)-algebra of Borel subsets of \( X \) with respect to some Polish topology on \( X \). The class of standard Borel spaces is closed under countable products, and a Borel subset of a standard Borel space is standard Borel when viewed as a subspace.
Every uncountable standard Borel space is in fact isomorphic to the \emph{Baire space}  \(\NN^{\NN}\) of
all functions from \(\NN\) to \(\NN\), with the Borel structure generated by the product topology.
We recall that this topology is generated by all sets \([s]=\setm{g\in \NN^{\NN}}{g\supseteq s}\) of end extensions of a given finite string \(s\).
We also define the set \( (\NN)^{\NN} \) as \( \setm{x\in \NN^{\NN}}{x\text{ is injective}} \), which is a closed subset of the Baire space \( \NN^{\NN} \) and therefore a Polish space with the induced topology.
Given any Polish space,  \( X \), the set \( F(X) \) of closed subsets of \( X \) is a standard Borel space when equipped with the \emph{Effros Borel structure}, namely,
the \( \sigma \)-algebra generated by the sets
\[
\setm{C\in F(X)}{C\cap U\neq\emptyset},
\]
where \( U \) is an open subset of \( X \) 
(see~\cite[Example~2.4]{Hjo} or \cite[Section~12.C]{Kec}).
A \textit{Polish group} is a topological group whose topology is Polish.
A well known example of a Polish group is \( S_{ \infty } \), the group of all bijections from \( \NN \) to \( \NN \). 
In fact, \( S_{\infty} \) is a \( G_{\delta} \) subset of the Baire space \( \NN^{\NN} \) and a topological group under the induced topology.
We define \( N_{s}\) as \([ s]\cap S_{\infty}\).
Note that the set \( \setm{N_{s}}{s \in {(\NN)}^{<\NN}}\) is a basis for \(S_{\infty}\),
where $(\NN)^{<\NN}$ denotes the set of finite sequences of 
distinct natural numbers.

A subset \(A\) of a standard Borel space \( X \) is  \emph{analytic}, or \(\analytic \),  if there is a Polish space \( Y \) and some Borel set \( B \subseteq X \times Y \) such that \( A \) is the 
projection
\[
p(B) = \setm{x \in X}{ \exists y \in Y((x,y) \in B)}.
\]
 A subset of a standard Borel space whose complement is analytic is called
 \textit{co-analytic}, or \( \coanalytic \). 
Souslin's Theorem (see~\cite[Theorem~14.11]{Kec}) states that the Borel sets of
a standard Borel space are precisely the sets that are both \(\analytic\)
and \(\coanalytic\).  

A function \(f\colon X\to Y\) between two standard Borel spaces \(X\) and \(Y\) 
is \emph{Borel} if the inverse image under \(f\) of any Borel set is Borel.
A corollary of Souslin's Theorem is that a function \(f\colon X\to Y\) between standard
Borel spaces is Borel if and only if
\(\setm{(x,f(x))\in X\times Y}{x\in X}\) is an analytic subset of $X\times Y$ (see~\cite[Theorem 14.12]{Kec}).

A \emph{quasi-order} is a reflexive and transitive binary relation. Any quasi-order \( Q \) on a set \( X \) naturally induces an equivalence relation \( E_{Q} \)
on \( X \) which is given by defining \( x \mathrel{E_Q} y \) if and only if \( x \mathrel{Q} y \) and \( y \mathrel{Q} x \).
In the cases considered in this paper, $Q$ will be the relation of embeddability
between structures, in which case $E_Q$ will be bi-embeddability,
a coarsening of the 
equivalence relation of isomorphism between structures.

A quasi-order \( Q \) on a standard Borel space \( X \) is a subset of \(X^{2}\) so we say that the quasi-order \(Q\) is
\emph{analytic} (resp. \emph{Borel}) 
if \( Q \)
is analytic (resp. a Borel) as a subset of \( X^{2} \) equipped with the product Borel structure. 
If \( Q \) is analytic (or Borel), then so is \( E_Q \).

If  \( \boldsymbol{G} \) is a Polish group and there is a Borel action \(a\) of \( \boldsymbol{G}\) on a standard Borel space \( X \), then we say that \( X \) is a \emph{standard Borel} \( \boldsymbol{G} \)\emph{-space} and
we denote by \( E_{a} \) the \emph{orbit equivalence relation} induced by that action. 
When the action is clear from the context we shall write \(E_{\boldsymbol{G}}^{X}\) instead of \(E_{a}\).
Such equivalence relations
are often called \( \boldsymbol{G} \)-\emph{equivalence relations}. Every \( \boldsymbol{G} \)-equivalence relation is analytic by definition and it is well known that all the classes of any $\boldsymbol{G}$-equivalence relation
are Borel (see~\cite[2.3.3]{BecKec}).
 The \textit{stabilizer} of a point \( x\) in \( X \) is the subgroup
\( \Stab{x} \coloneqq  \setm{g\in \boldsymbol{G}}{g\cdot x=x}, \)
where \( g \cdot x \) denotes the value of the action on the pair \( (g,x) \). We will use the fact that each stabilizer is a closed subgroup of \(\boldsymbol G\) (see \cite[9.17]{Kec}), and that
the set \( \Subg{\mathbf{G}} \) of closed subgroups of \( \mathbf{G} \) is a Borel subset of \( F(\mathbf{G}) \). Thus \( \Subg{\mathbf{G}} \) is standard Borel space with the induced Borel structure.

 In this paper we focus mainly on standard Borel spaces of countable structures.
If \( L \) is a countable (relational) language we denote by  \( X_{L} \)
the \emph{space of \( L \)-structures with domain \( \NN \)}, whose topology is the one defined by taking as basic open sets those of the form
\[
\setm{\mathcal{M}\in X_{L}}{\mathcal{M}\models R(n_{0},\dots ,n_{k-1})},\quad
\setm{\mathcal{M}\in X_{L}}{\mathcal{M}\models \lnot R(n_{0},\dots ,n_{k-1})},
\]
for any \(k\)-tuples \( (n_{0},\dots , n_{k-1})\) of natural numbers and any relation \( R \) in \( L \) of arity \(k=a(R)\). Such a space is Polish because it is homeomorphic to
\( \prod_{R \in L} 2^{\NN^{a(R)}} \). (An analogous definition can be given also for languages with function symbols, see~\cite[Section~2.5]{BecKec}.) Let \(S_{\infty} \) act on \( X_{L} \) continuously by the so-called \emph{logic action}:
 for every \(g\) in \( S_{\infty}\) and  \(\mathcal{M}\), 
\(\mathcal{N}\in X_{L}\) we set
 \(g\cdot \mathcal{M}=\mathcal{N}\) if for all \(k\)-ary relations \( R\) in \( L \) and all \(k\)-tuples of natural numbers \((n_{0},\dotsc, n_{k-1})\), we have
\[
\mathcal{N} \models R(n_{0},\dots, n_{k-1}) \iff \mathcal{M} \models R(g^{-1}(n_{0}),\dots, g^{-1}(n_{k-1})).
\]
In other words, the structure \(g\cdot \mathcal{M}\) is obtained by interpreting each relation symbol as in \(\mathcal{M}\) up to \(g\), which is a permutation of natural numbers.
Thus, for any countable \( L \), the space \( X_{L} \) is a standard Borel \( S_{\infty}\)-space; and the isomorphism relation on \(X_{L}\), usually denoted by \(\iso_{L}\), coincides with the orbit equivalence relation \( E_{S_{\infty}}^{X_{L}} \). Moreover notice that, for every \( \mathcal{M}\) in \( X_{L} \), we have equality between 
\(
\Stab{ \mathcal M }\) and the group of automorphisms of \(\mathcal{M}\), \(\Aut{ \mathcal M }
\).

Given two quasi-orders \( P \) and \( R \) on the standard Borel spaces \( X \) and \( Y \), respectively, we say that \( P \) \emph{Borel reduces} (or is \emph{Borel reducible}) to \( R \), written \( P \leq_B R \), if and only if there is a Borel function \( f \colon X\to Y\) such that for every \( x,y\) in \( X \)
\[
x \mathrel{P} y \iff f(x)\mathrel{R}f(y).
\]
Such an \( f \) is called a \emph{Borel} \emph{reduction}.
We say that \(P\) is \emph{essentially \(R\)}, denoted \( P \sim_B R \), whenever \(P\) and \( R \) are \emph{Borel bi-reducible}: that is, \( P \leq_B R \) and \( R \leq_B P \).

Louveau and Rosendal proved in \cite{LouRos} that among all \( \analytic \) quasi-orders there are \( \leq_B \)-maximum elements called \emph{complete \( \analytic \) quasi-orders}. One of the most prominent examples of such a maximum element is the quasi-order of embeddability between \emph{combinatorial trees}.
By a \emph{graph} we mean a structure for an irreflexive and symmetric binary relation symbol called the \emph{edge relation}.
A combinatorial tree is a connected acyclic graph.

Let \( X_\mathrm{Gr} \) be the space of graphs on \( \NN \).  Identifying
each graph with the characteristic function of its edge relation as above, \( X_\mathrm{Gr} \) is a closed subset of \( 2^{\mathbb{N}^2} \), and thus is a Polish space. 
We denote by \( X_{\mathrm{CT}} \) the set of combinatorial trees with vertex set \( \mathbb{N} \), and note that \(X_{\mathrm{CT}}\) is a \(G_{\delta}\) subset of \( X_\mathrm{Gr} \)
(towards this, first observe that the set of graphs with a path from $m$ to $n$
is open for all $m$ and $n$ in $\NN$). 
Hence, $X_{\mathrm{CT}}$ is a Polish space with the induced topology
(see for example \cite[Theorem~3.11]{Kec}).
For graphs \( S,T\) in \( X_\mathrm{Gr} \), we say that \( S \) \emph{embeds}, or \(S\) \emph{is embeddable} into \( T \),
\( S\embeds_\mathrm{Gr}T \), if and only if
there is a one-to-one function \( f \colon \NN \to \NN\) which realizes an isomorphism between \( S \) and \( T\upharpoonright \mathop{\mathrm{Im}}(f) \).
The quasi-order \( \embeds_\mathrm{Gr} \)  is analytic because it is the set
\[
\setm{(S,T)\in (X_\mathrm{Gr})^{2}}{\exists f\in {(\mathbb N)}^{\mathbb N}(\forall n,m\in\mathbb{N}( (n,m)\in S)\iff (f(n),f(m))\in T) )},
\]
which is a projection of a closed subset of \( {\NN}^{\NN}\times X_\mathrm{Gr}\times X_\mathrm{Gr} \).
 We denote by \( \embeds_{\mathrm{CT}} \) the restriction of the quasi-order \(\embeds_\mathrm{Gr}\) to \(X_{\mathrm{CT}}\).

\begin{theorem}[{\cite[Theorem 3.1]{LouRos}}]\label{Theorem : LouRos}
The relation \( \embeds_{\mathrm{CT}} \) of  embeddability between countable combinatorial trees is a complete \( \analytic \) quasi-order.
\end{theorem}

All trees built in the proof of Theorem~\ref{Theorem : LouRos} satisfy the further property that there are no complete vertices,  expressible by the formula:
\begin{equation} \label{eq : property}\tag{\(\sqcup\)}
\forall x \exists y (x \neq y\land (x,y)\notin T).
\end{equation}
We denote by \(X_{\mathrm{CT}^{\sqcup}}\) the standard Borel space of combinatorial trees satisfying \eqref{eq : property}.
In \cite[Section~2]{FriMot} and \cite[Section~3]{CamMarMot}, the authors modified the proof of Theorem~\ref{Theorem : LouRos} to  prove the following proposition.

\begin{proposition}\label{proposition : FriMot}
There is a Borel \(\XX\subseteq X_{\mathrm{CT}^{\sqcup}}\) such that:

\begin{enumerate-(i)}
\item \label{property1} the equality and isomorphism relations restricted to \( \XX\), denoted respectively by \( =_{\XX} \) and \( \iso_{\XX} \), coincide;
\item \label{property2} each graph in \( \XX \) is rigid; that is,  it has no nontrivial automorphism;
\item \label{property3} for every \(\analytic\) quasi-order \(P\) on \( 2^{\NN}\), there exists an injective Borel reduction \( \alpha\mapsto T_{\alpha} \) from \(P\) to \(\embeds_\XX\).
\end{enumerate-(i)}
\end{proposition}

This result is a strengthening of Theorem \ref{Theorem : LouRos}. 
A closer look into \cite{CamMarMot} shows that the map \(\alpha\mapsto T_{\alpha}\) in \ref{property3} of Proposition~\ref{proposition : FriMot} is constructed by first reducing \(P\) to a quasi-order, which is denote by \(\leq_\text{max}\) and is defined on the standard Borel space \( \mathcal{T}\) of normal trees\footnote{The precise definition of \(\leq_\text{max}\) is not relevant to the results of this paper. We refer the interested reader to \cite[Definition~2.3]{LouRos}.} on \( 2\times\omega \),
and then reducing \(\leq_\text{max}\) to \(\embeds_{\mathrm{CT}^{\sqcup}}\).
Both those reductions are injective. Next one defines \(\XX\) as the image of the whole of \(\mathcal{T} \) through the second map. Clearly, \(\XX\) is a Borel subset of \(X_{\mathrm{CT}^{\sqcup}}\) as it is the injective image of a standard Borel space through a Borel map~\cite[Corollary~15.2]{Kec}. Moreover, since \(\leq_\text{max}\) is known to be a complete \( \analytic \) quasi-order (see \cite[Theorem~2.5]{LouRos}),
so is the quasi-order \(\embeds_{\XX}\).
Therefore in contrast to items \ref{property1} and \ref{property2},
the bi-embeddability relation on $\XX$ will be highly nontrivial,
and the graphs in $\XX$ will have many nontrivial endomorphisms.


\section{Invariant universality}\label{sec : invuniversality}

The property of invariant universality (Definition~\ref{Definition : invariantly universal}) was first observed in \cite{CamMarMot} for embeddability between countable combinatorial trees when the equivalence relation is isomorphism.

 \begin{definition}[{\cite{CamMarMot}}]\label{Definition : invariantly universal}
Let \( P \) be a \( \analytic \) quasi-order on some standard Borel space \( X \) and let \(E\) be a \( \analytic \) equivalence subrelation of \( P \). We say that \( (P,E) \)
is \emph{invariantly universal} (or \( P \) is invariantly universal with respect to \( E \)) if for every \( \analytic \) quasi-order \( R \) there is a Borel subset \( B\subseteq X \) which is invariant with respect to \( E \) and
such that  \( P \restriction B \) is essentially \( R \).
 \end{definition}
 
When we look at relations defined on a space of countable structures, if \((P,E)\) are as in Definition~\ref{Definition : invariantly universal} and \(E\) is the relation of isomorphism, we simply say that \(P\) is invariantly universal.
By a classical result of Lopez-Escobar (see \cite[Theorem~16.8]{Kec}), a subset of a space of countable structures is closed under isomorphism if and only if it is definable in the logic \(\mathcal{L}_{\omega_{1}\omega}\). Examples of invariantly universal quasi-orders found in \cite{CamMarMot,CalMot,CamMarMot17} include: linear isometric embeddability between separable Banach spaces;
 embeddability between countable groups; and isometric embeddability on ultrametric Polish spaces with any prescribed ill-founded set of distances.
 
The standard Borel space \( \XX \) defined in Section~\ref{sec : preliminaries} is used to test whether a pair \( (Q,E) \) satisfying the hypotheses of Definition \ref{Definition : invariantly universal} is invariantly universal. The following result, which is essentially a particular case of~\cite[Theorem 4.2]{CamMarMot},
gives a sufficient condition for the invariant universality of a pair.

\begin{theorem}[{\cite{CamMarMot}}]\label{Theorem : CMMR13}
Let \( P \) be a \( \analytic \) quasi-order on a 
space \(X_L\) of \(L\)-structures with domain \(\NN\)
such that \( \iso_{L}\subseteq P \).
Suppose that the following conditions hold:
\begin{enumerate-(i)}\label{enumerate : TheoremCMMR13}
 \item \label{condition : 1} there is a Borel reduction \( f \colon \XX\rightarrow X_{L} \) of \( \embeds_\XX \) to \( P \);
\item \label{condition : 2} \( f \) is also a Borel reduction of \( {=_\XX} \) (equivalently, of \( \cong_{\XX} \)) to \( \iso_{L} \);
\item \label{condition : 3} the map
\(
 \XX\to\Subg { S_{\infty}}, T\mapsto \Stab{f(T)}=\Aut{f(T)}
\)
 is Borel.
\end{enumerate-(i)}
Then, for every \( \analytic \) quasi-order \( R \) there is a
Borel \( B\subseteq X_{L} \) such that \( R \) is essentially \( P \restriction B\).
\end{theorem}

One of the open questions about invariant universality in the paper by Camerlo, Marcone, and Motto~Ros is the following.

\begin{question}[{\cite[Question~6.3]{CamMarMot}}]\label{CMMRQ63}
Is there a natural pair \((P,E)\) which is not invariantly universal 
but for which \(P\) is a complete analytic quasi-order?
\end{question}

We stress the word ``natural'' --- although examples of such pairs are known, none of them consists of relations that arise in other contexts defined over a space of mathematical objects. 
Our results show that the specific examples of quandle embedding and of field 
embedding for fields of characteristic not equal to 2 (each with the equivalence
relation of isomorphism) do not furnish examples for an affirmative answer
to Question~\ref{CMMRQ63}.


\section{Quandles and related structures} \label{sec : quandles}
In this section we use the reduction from graphs to quandles defined in \cite{BroMil} to prove that embeddability between countable quandles is a complete \( \analytic \) quasi-order.
Recall that a set \( Q \) with a binary relation \(\ast\) is a \emph{quandle} if:
\begin{enumerate-(a)}
\item \label{ax:quandle1} \( \forall x,y,z \in Q(x\ast(y\ast z)=(x\ast y)\ast(x\ast z))\);
\item \label{ax:quandle2} \( \forall x,z\in Q \ \exists ! y\in Q (x\ast y=z)\);
\item \label{ax:quandle3} \( \forall x\in Q (x \ast x=x)\).
\end{enumerate-(a)}
For an introduction to the theory of quandles, see for example \cite{ElhNel}.

We now recall the reduction appearing in \cite{BroMil}. For any \( T \) in \( X_\mathrm{Gr} \),
let \( Q_{T} \) be the quandle with underlying set \(\NN \times \set{0,1}\) and the binary operation be \(\ast_{T} \) defined as follows:
\begin{equation}\label{eq : defast}\tag{\(\ast\)}
(u,i) \ast_{T} (v,j)=
\begin{cases}
(v,j) & \text{if \( u = v \) or \( (u,v) \in T \)},\\
(v,1-j) & \text{otherwise}.
\end{cases}
\end{equation}
It is straightforward to check that \( (Q_{T}, \ast_{T} )\) satisfies \ref{ax:quandle1}--\ref{ax:quandle3}.
In the sequel, we denote  the space of quandles with domain \( \NN \) by \( X_\mathrm{Qdl} \),
which is a \( G_{\delta} \) subset of \( 2^{\NN^{3}} \) and thus a Polish space. For every graph \(T\) in \( X_\mathrm{Gr} \), the quandle \( Q_{T} \)
can be easily coded as an isomorphic structure 
\( \mathcal{Q}_{T} \) with domain \( \NN \), for example use the bijection \( \NN\times 2 \to \NN, (n,i)\mapsto 2n+i \).
Clearly the map $T\mapsto \mathcal{Q}_T$ is Borel; in fact, it is continuous. Recall the following definition.

\begin{definition}
Suppose that there is a Borel action \(a\) of \(S_{\infty}\) on some standard Borel space and \(E=E_{a}\).
We say that \(E\) is \emph{\(S_{\infty}\)-complete} if every equivalence relation induced by a Borel action of \( S_\infty \) on
some standard Borel space Borel reduces to \(E\).
\end{definition}

The main theorem of \cite{BroMil} is the following.

\begin{theorem}[{\cite[Theorem 3]{BroMil}}]\label{Theorem : BroMil} For all graphs \( S,T \) in \( X_\mathrm{Gr} \), we have
\[
S \iso_\mathrm{Gr} T\qquad  \iff \qquad Q_{S} \iso_\mathrm{Qdl} Q_{T}.
\]
Thus, the equivalence relation of isomorphism on the space
of countable quandles is \( S_{\infty} \)-complete.
\end{theorem}

Proving that \( S \iso_\mathrm{Gr} T \) implies \( Q_{S} \iso_\mathrm{Qdl} Q_{T} \) is straightforward but the converse is considerably more involved. In the proof of  Theorem~\ref{Theorem : BroMil}, whenever
\(S\) contains complete vertices and  \( \rho \) is an isomorphism from \( Q_{S} \)
to \(Q_{T}\),
the surjectivity of \( \rho \) is used substantially to recover an isomorphism of graphs between \( S \) and \( T \). Since embeddings do not need to be surjective, we cannot prove an analog of Theorem~\ref{Theorem : BroMil} in the same way. However, if we restrict our attention to \(X_{\mathrm{CT}^{\sqcup}}\), a simpler argument allows us to prove 
Theorem~\ref{Theorem : embeddability}.

Towards this we now analyze quandle embeddings. We recall the following fact from \cite{BroMil}.

\begin{lemma}[{\cite[Lemma 1]{BroMil}}]
For every \( T\) in \( X_\mathrm{Gr} \) and every \( A\subseteq \mathbb{N}\), the function \(I_{A} \colon Q_{T} \to Q_{T} \) defined by
\[
I_{A} (v, j) =
\begin{cases}
  (v, j) & \text{ if \( v \in A \)}\\
  (v,1-j) & \text{ otherwise}
\end{cases}
\]
is an involution of \( Q_{T}\).
\end{lemma}

For any quandle homomorphism $\rho\colon Q_S\to Q_T$ between quandles derived from
graphs, let us denote by \( \rho_{V}(v,i) \) and \( \rho_{I}(v,i)\) the first and the second components of \( \rho(v,i) \), respectively.

\begin{lemma}\label{claim : no fixed}
For every graph \( S \) satisfying \eqref{eq : property} and graph $T$,
every vertex \( v\) of  \(S \), and every quandle homomorphism 
$\rho\colon Q_S\to Q_T$,
\[\rho_{V}(v,0)=\rho_{V}(v,1).\]
\end{lemma}
\begin{proof}
Since \( S\) satisfies \eqref{eq : property}, for every vertex \(v\) of \( S\)
there is another vertex \( v^{+} \) such that
\( v\) and \(v^{+}\) are not adjacent in \( S \). Then, by applying \(\rho\) to both sides of
\[
(v^{+},0) \ast_{S} (v,0)  =  (v,1)
\]
we get
\(
\rho(v^{+},0) \ast_{T} \rho(v,0)  = \rho(v,1),
\)
which implies that \( \rho_{V}(v,0)  = \rho_{V}(v,1)\) by definition (see \eqref{eq : defast}).
\end{proof}

With these ingredients we can present a factorisation lemma for the quandle
homomorphisms we shall be interested in.  We thank the anonymous referee for
their suggestion to streamline our results in this way.

\begin{lemma}
\label{lemma:factorization}
Let $S$ and $T$ be graphs satisfying \eqref{eq : property}.
Every embedding \( \rho\colon Q_{S}\to Q_{T} \) is obtained in the following manner: there is some graph embedding \( h\colon S \to T \) and some \( A\subseteq \NN \) such that
\[
\rho(v,j)=I_{A}(h(v),j).
\]
\end{lemma}

\begin{proof}
Assume that $S$ and $T$ are graphs satisfying \eqref{eq : property}, and that
\( \rho \colon Q_{S} \to Q_{T} \) is a quandle embedding. 
We define
\[
h \colon S\to T,\qquad v \mapsto \rho_{V}(v,0)=\rho_{V}(v,1).
\]
Lemma~\ref{claim : no fixed} shows that $h$ is well-defined.
Next we show that \(h\) is injective. The equality \(h(v)=h(w)\) implies that 
\[\rho_{V}(v,0)=\rho_{V}(v,1)=\rho_{V}(w,0)=\rho_{V}(w,1),\]
which implies in turn that \(\rho(v,0)=\rho(w,i)\) for either \(i=0\) or \(i=1\). By injectivity of \(\rho\), we get \( i=0\) and \(v=w\).
It remains to show that \( h \) is a graph embedding.
Pick any two adjacent vertices \(u\) and \(v\) in \(  S \). Notice that \(u\) and \(v\) are necessarily distinct and \( \rho(u,0)\ast_{T} \rho(v,0)=\rho(v,0) \). 
So either \(\rho_{V}(u,0)=\rho_{V}(v,0)\) or \((\rho_{V}(u,0),\rho_{V}(v,0))\in T\).
The first cannot hold by injectivity of $h$ just shown.
Thus it is the case that
\[
(h(u),h(v))=(\rho_{V}(u,0), \rho_{V}(v,0)) \in T .
\]
On the other hand, if \( (u,v)\notin S \) then
 \((v,j) \ast_{S} (u,0)=(u,1) \). By applying \(\rho\) to both terms, we get 
\( \rho(v,j) \ast_{T} \rho(u,0)=\rho(u,1) \). By Lemma~\ref{claim : no fixed}, we have that
\(\rho_{V}(u,0)\) equals \(\rho_{V}(u,1)\), so necessarily \( \rho_{I}(u,0)\neq\rho_{I}(u,1)\) because \(\rho\) is injective. Then, by definition of \(\ast_{T}\) we have
\[
(h(u),h(v))=(\rho_{V}(u,0), \rho_{V}(v,j)) \notin T.
\]
So $h$ is a graph embedding.

To complete the proof of Lemma~\ref{lemma:factorization} let
\[
A = \setm{n \in \NN}{\rho_{I}(n,i)\neq i}.
\]
By construction we obtain that \(\rho(v,j)=I_{A}(h(v),j)\) for every \((v,i)\in Q_{S}\).
\end{proof}

\begin{theorem}\label{Theorem : embeddability}
The relation \( \embeds_\mathrm{Qdl} \) of embeddability on the space of countable quandles
is a complete \( \analytic \) quasi-order.
\end{theorem}
\begin{proof}
It suffices to prove that \( \embeds_{\mathrm{CT}^{\sqcup}} \) Borel reduces to \( \embeds_\mathrm{Qdl} \). We show that the map from
\(X_{\mathrm{CT}^{\sqcup}}\) to \( X_\mathrm{Qdl}\) taking \( T\) to \( Q_{T} \) is a reduction.
Assume that \( f\colon S\to T \) is a graph embedding, then consider the function
\(\theta\colon  Q_{S}\to Q_{T}\) such that \((v,i)\mapsto (f(v),i)\).
Injectivity of \(\theta\) is immediate. Moreover,
for all \( (u,i)\) and \((v,j)\) in \( Q_{S} \),
\[
\theta((u,i)\ast_{S}(v,j))=
\theta(u,i)\ast_{T} \theta(v,j).
\]
In fact, by applying the definitions of \( \theta \) and \( \ast_{S} \), we have
\[
\theta((u,i)\ast_{S}(v,j))=
\begin{cases}
(f(v),j) & \text{if \( u = v \) or \( (u,v) \in T \)},\\
(f(v),1-j) & \text{otherwise};
\end{cases}
\]
and the first condition is equivalent to 
\( f(u) = f(v) \text{ or } (f(u),f(v)) \in T \)
because \( f \) is a graph embedding. Therefore, \( \theta \) witnesses that \( Q_{S}\) is embeddable into \( Q_{T} \).

Now the converse is a straightforward consequence of Lemma~\ref{lemma:factorization}. Whenever \( \rho \colon Q_{S} \to Q_{T} \) is a quandle embedding we recover a graph embedding \( h \colon S\to T \) such that
\(\rho(v,j)=I_{A}(h(v),j).\) Therefore, \(S\embeds_\mathrm{Gr} T\) as desired.
\end{proof}

Before proving the main result of this section we isolate a particular case of Lemma~\ref{lemma:factorization}.

\begin{lemma}\label{Lemma : auto}
Let $T$ be a graph satisfying \eqref{eq : property}.
Every \( \rho \) in  \( \Aut{Q_{T}} \) is obtained from some graph automorphism \( h \) in  \( \Aut{T} \) in the following manner: there is an \( h\) in \(\Aut{T} \) and some \( A\subseteq \NN \) such that
\[
\rho(v,j)=I_{A}(h(v),j).
\]
\end{lemma}
\begin{proof}
Every automorphism \(\rho\) of \(Q_{T}\) is in particular a self-embedding of \(Q_{T}\).  Hence by Lemma~\ref{lemma:factorization}, there is an embedding \(h\colon T\to T\) such that \(\rho(v,j)=I_{A}(h(v),j)\). By construction \(h\) is surjective.
\end{proof}

We recall that \( \mathcal{Q}_{T} \) is defined as the quandle with domain \(\NN \) which is isomorphic to \( Q_{T} \) via the bijection
\( \NN\times 2 \to \NN\) taking \((n,i)\) to \( 2n+i \).

\begin{theorem}\label{Theorem : main}
The relation \( \embeds_\mathrm{Qdl} \) of embeddability between countable quandles is an invariantly universal  \( \analytic \) quasi-order.
\end{theorem}
\begin{proof}
By Theorem~\ref{Theorem : CMMR13} it suffices to prove that \( \embeds_\mathrm{Qdl} \) and
\( \iso_\mathrm{Qdl} \) together satisfies \ref{condition : 1}--\ref{condition : 3}.
Let \(f\) be the map from \(\XX\) to \(X_\mathrm{Qdl}\) taking \(T\) to \( \mathcal{Q}_{T}\).
By Theorem~\ref{Theorem : embeddability} \(f\) Borel reduces \( \embeds_{\XX}\) to \( \embeds_\mathrm{Qdl}\), and by Theorem~\ref{Theorem : BroMil} we know that \( \iso_{\XX} \) Borel reduces to \( \iso_\mathrm{Qdl} \) via the same map, hence \ref{condition : 1} and \ref{condition : 2} hold.

By Lemma~\ref{Lemma : auto}, whenever \( \rho\) is in \( \Aut{\mathcal{Q}_{T}} \) there exist
 some \( h\) in \(\Aut{ T} \)  and some \( A \subseteq \mathbb{N} \) such that
\( \rho(v,j) = I_{A} ( h(v),j) \). Further, since each \( T\) in \( \XX \) is rigid, we have \( h=id \) and consequently \( \rho=I_{A} \) for some \( A \subseteq \mathbb{N} \).
Thus for every \(T\) in \(\XX\),  \( g\) is an automorphism of \( \mathcal{Q}_{T} \) if and only if there is some
\(A\subseteq \mathbb{N} \) such that for \(i\in\set{0,1}\)
\[
g(2v+i)=
\begin{cases}
2v+i&v\in A\\
2v+1-i&\text{otherwise.}
\end{cases}
\]
Note in particular that this depends only on $A$, not on $T$.

To see that the \( T\mapsto \Aut{\mathcal{Q}_{T}} \) is Borel it suffices to show that the preimage of every basic open set is Borel.   For every fixed \( s\) in \( (\NN)^{<\NN} \), the preimage of
\[
\setm{G\in \Subg{S_{\infty}}}{ G\cap N_{s}\neq\emptyset}
\]
through the map \( T\mapsto \Aut{\mathcal{Q}_{T}}\)
is the set 
\[
\setm{T\in \XX}{ \Aut{\mathcal{Q}_{T}}\cap N_{s}\neq\emptyset}
=
\begin{cases}
\XX & \pbox[t]{2.5in}{if every \(n\) in \(\dom s\) is either sent to itself or,
						if not, swapped with
						its successor if $n$ is even
					and predecessor if \(n\) is odd,}\\
\emptyset & \text{otherwise},
\end{cases}
\]
which is certainly a Borel set.
\end{proof}

\begin{corollary}
For every \( \analytic \) quasi-order \( R \) there is an \( \L_{\omega_1 \omega} \)-elementary class \(B\) of countable quandles such that the embeddability relation on \(B\) is Borel bi-reducible with \( R \).
\end{corollary}

In \cite{BroMil} other quandle-like structures are considered. A quandle is a \emph{kei}
if and only if it satisfies
\[
\forall x \forall y (x\ast(x\ast y)=y).
\]

It is easy to check that for every \(T\) in \( X_\mathrm{Gr}\), \( Q_{T}\) defined as in Section~\ref{sec : invuniversality} is a kei. Therefore, arguing as in Theorem~\ref{Theorem : main} one can prove the following.

\begin{theorem}
The embeddability relation between countable kei is invariantly universal. 
\end{theorem}

\begin{definition} 
An \emph{LD-monoid}, or \emph{algebra satisfying \(\Sigma\)}, is a structure over the language \(\set{\ast,\circ}\) consisting of two binary operational symbols satisfying for all \(a,b,c\) the following identities
\begin{align*}
a\circ(b\circ c) & = (a\circ b)\circ c,\\
(a\circ b)\ast c & = a\ast(b\ast c),\\
a\ast(b\circ c) & = (a\ast b)\circ (a\ast c),\\
(a\ast b)\circ a & =a\circ b.
\end{align*}
\end{definition}

The terminology ``LD-monoid'' was introduced by Dehornoy, while Laver called such structures ``algebras satisfying \(\Sigma\).''
Notice that if \( (M,\circ_{M}) \) is a group and \(\ast_{M}\) is the conjugation operation on \(M\),
\begin{equation*}\label{eq : conjugation}
a\ast_{M} b = a\circ_{M} b\circ_{M} a^{-1},
\end{equation*}
then \((M, \circ^{M}, \ast^{M})\) is an LD-monoid.

In \cite[Theorem 4]{BroMil} the authors observed that the equivalence relation of isomorphism between LD-monoids is \( S_{\infty} \)-complete.

\begin{theorem}
The quasi-order of embeddability between countable LD-monoids is invariantly universal.
\end{theorem}
\begin{proof}
In \cite{Wil14} J. Williams defines a Borel reduction 
\(h\colon X_\mathrm{Gr}\to X_\mathrm{Gp}\)
from \(\embeds_\mathrm{Gr}\) to \(\embeds_\mathrm{Gp}\).
Then in~\cite[Theorem 3.5]{CalMot} the second author and Motto~Ros observe that: 
\begin{enumerate-(a)}
\item
\label{(a)}
\(h\restriction \XX\) is a Borel reduction from \( =_{\XX}\) to \(\iso_\mathrm{Gp}\), and
\item
\label{(b)}
the map \( \XX\to \Subg{S_{\infty}}\) sending \(T\) to \( \Aut{h(T)} \) is Borel.
\end{enumerate-(a)}

For any $G=(\NN,\circ_G)$ in $X_\mathrm{Gp}$,
let \(M(G)=(\NN,\circ_{G},\ast_{G})\) be the LD-monoid over \(\mathbb{N}\) such that \( \ast_{G} \) is interpreted as the conjugation operation in \((\NN,\circ_{G})\).
That is, we define
\({\ast_{G}}\colon \NN\times \NN \to \NN\) by
\[
k\ast_{G} m =  k \circ_G m \circ_G k^{-1}.
\]
The LD-monoid \(M(G)\) is thus \(G\) with enriched structure ---
in model-theoretic terms, it is a definitional expansion of $G$.
Now we observe that this map $M\colon X_\mathrm{Gp}\to X_\mathrm{LD-m}$ to the 
space $X_\mathrm{LD-m}$ of LD-monoids is a Borel reduction, 
reducing group embeddability to the embeddability relation between LD-monoids,
and reducing group isomorphism to LD-monoid isomorphism.
Indeed, any homomorphism \(\phi\colon G\to H\) (resp. embedding) realizes a homomorphism (resp. embedding) between the corresponding \(M(G)\) and \(M(H)\) as
 \begin{align*}
 k\ast_{G}m=n\quad & \iff\quad k\circ_{G}m =  n\circ_{G}k \\
 & \iff \quad
 \phi(k)\circ_{H} \phi(m) =  \phi(n)\circ_{H}\phi(k) \\
& \iff \quad\phi(k)\ast_{H}\phi(m)=\phi(n).
 \end{align*}
Conversely, it is clear that any homomorphism (resp. embedding) from \(M(G)\) to \(M(H)\) gives a group homomorphism (resp. embedding) \(G\to H\) between the underlying group structures by simply ``forgetting'' the $\ast$ operation.

If we let \(f\colon\XX\to X_\mathrm{LD-m}\) be the composition $M\circ h$, verifying conditions \ref{condition : 1}--\ref{condition : 3} of Theorem~\ref{Theorem : CMMR13} will give the invariant universality of bi-embeddability on LD-monoids, as desired.
 Since \(h\) Borel reduces \(\embeds_\mathrm{Gr}\) to \(\embeds_\mathrm{Gp}\), it follows that \(f\) Borel reduces  \(\embeds_\mathbb{X}\) to embeddability on LD-monoids; hence \ref{condition : 1} holds. Condition~\ref{(a)} implies that \(f\) is a reduction from \(\cong_\mathbb{X}\) to isomorphism on LD-monoids; hence we get \ref{condition : 2}.  Finally, since \(\Aut{M(h(T))}=\Aut{h(T)}\), condition \ref{(b)} ensures that the map \(T\mapsto \Aut{M(h(T))}\) is Borel, which gives condition \ref{condition : 3} of Theorem~\ref{Theorem : CMMR13}.
\end{proof}

\section{Fields}\label{sec : fields}

We denote by \(X_{\mathrm{Fld},p}\) the standard Borel space of fields of fixed characteristic \(p\).
The relation of isomorphism on \(X_{\mathrm{Fld},p}\) is an \(S_{\infty}\)-complete equivalence relation for every characteristic $p$ --- see
\cite[Theorem~10]{FriSta} and \cite{Sha}. 
In this section we study the quasi-order of embeddability on \(X_{\mathrm{Fld},p}\), which we denote by \(\embeds_{\mathrm{Fld},p}\). Recall that, since any field has only trivial ideals, every field homomorphism is one-to-one, and thus the notions of embeddability and homomorphism coincide. Therefore we adopt the usual terminology from algebra that if \(f\colon F\to L\) is a homomorphism of fields we say that \(F\) is a subfield of \(L\), or that \(L\) is a field extension of \(F\).

If  \(F\) is a field and \(S\) is a set of algebraically independent elements over \( F \), we denote by \( F(S)\) the purely trascendental extension of \(F\) by \(S\).
If \(S\) is a singleton, \(\set{s}\), we write \(F(s)\) instead of \(F(\set{s})\). Following the notation of \cite{FriKol}, for any prime
\(p\),
 any field \(F\), and any  set \(S\) of algebraically independent elements over \(F\), we denote by \(F(S)(S,p)\) the smallest field extension of \(F(S)\) containing \(\setm{s(n)}{ s \in S,n< \omega}\), where
\begin{itemizenew}
\item \( s(0) = s \),
\item \( s(n + 1) \) is such that \( (s(n + 1))^{p}= s(n) \). 
\end{itemizenew}
Notice that this uniquely determines \(F(S)(S,p)\) up to isomorphism. 
We use the convention \(F(s)(s,p)=F(\set{s})(\set{s},p)\).

We now recall a construction of Fried and Koll\'ar~\cite{FriKol} that, given a combinatorial tree \( T \) of infinite cardinality, produces a field  \(K_{T}\), and furthermore this construction respects embedding.
For clarity we denote by \(V=\set{v_{0},v_{1},\dots}\) the set of vertices of the graphs in \(X_{\mathrm{CT}}\).
\begin{definition}[{\cite[Section~3]{FriKol}}]
 Fix a characteristic \(p\) equal to \(0\) or an odd prime number, 
fix \(F\) a countable field of characteristic \(p\),
and fix  an increasing sequence of odd prime numbers \( \setm{p_{n}}{n\in\NN} \) not containing \(p\).
For any \( T\) in \( X_{\mathrm{CT}} \),  we define \( K_{T} \) as the union of an increasing chain of fields
\( K_{n}(T) \).
These fields \( K_{n}(T) \) are defined recursively. First define
\[
K_{0}(T)\coloneqq F(V)(V,p_{0}) \quad\text{and}\quad  H_{0}(T)\coloneqq\setm{u+v}{(u,v)\in T}.
\]
Next suppose that \(K_{n}(T)\) and \(H_{n}(T)\) have already been defined. Fix a trascendental element \(t_{n}\) over \(K_{n}(T)\), and let
\(L_{n}\) be  the field \(K_{n}(T)(t_{n})(\set{t_{n}},p_{n+1})\).
Now we define
\(K_{n+1}(T)\) as the splitting field over \( L_{n} \) of the set of polynomials
\[
P_{n}=\setm{x^{2}-(t_{n}-a)}{a\in H_{n}(T)}.
\]
Further, we define \(H_{n+1}(T)\) to be a set containing exactly one root of each of the polynomials in \(P_{n}\).
\end{definition}

The next two lemmas summarize the essential properties of the map sending any \(T\) of \( X_{\mathrm{CT}}\) to \(K_{T}\).
They were implicitly obtained in the paper of Fried and Koll\'ar~\cite{FriKol}.

\begin{lemma}\label{Lemma : K homomorphism} If there is a graph embedding from
\(S\) to \(T\), then \( K_{S}\) is a subfield of \( K_{T}\).
\end{lemma}
In fact, Fried and Koll\'ar~\cite{FriKol} proved inductively that if there is a graph embedding from
\(S\) to \(T\), then each \(K_{n}(S)\) is a subfield of \( K_{n}(T)\).

\begin{lemma}\label{Lemma : preserve edges}
Let \( \phi \colon K_{S} \to K_{T} \) be a homomorphism.
\begin{enumerate-(a)}
\item
For every \(n\) in \(\mathbb{N}\), \(\phi\) maps \( H_{n}(S)\) into \(H_{n}(T)\). In particular, we have \( \phi[H_{0}(S)]\subseteq H_{0}(T)\).
\item
\label{item:b}
Suppose that \(u\) is a vertex of \(S\). If \(u\) is not isolated and \((u,v)\) is an edge in \(S\), then
\(\phi(u)\) is in \( V\) and \((\phi(u),\phi(v))\) is an edge in \(T\).
\end{enumerate-(a)}
\end{lemma}

The next theorem is a consequence of the previous two lemmas.
The structure of the proof is as for \cite[Theorem~2.1]{FriKol}, 
but since we are concerned only with the embeddability relation rather than
all embeddings, we are able to include the odd characteristic case, unlike that
theorem.

\begin{theorem}\label{Theorem : field complete}
For every \(p\) equal to \(0\) or an odd prime number,
the quasi-order \(\embeds_{\mathrm{CT}}\) Borel reduces to \(\embeds_{\mathrm{Fld},p}\). Thus \(\embeds_{\mathrm{Fld},p}\) is a complete \( \analytic \) quasi-order.
\end{theorem}
\begin{proof}
The map taking each \(T\) in \( X_{\mathrm{CT}}\) to \(K_{T}\) can be realized as a Borel map from \(X_{\mathrm{CT}}\) to \(X_{\mathrm{Fld},p}\).
If \(S\) is embeddable into \( T\), then \(K_{T}\) is a field extension of \(K_{S}\) by Lemma~\ref{Lemma : K homomorphism}.
Now suppose that \(\rho\colon K_{S}\to K_{T} \) is a homomorphism. We claim that \(f\) defined as the restriction map \(\rho\restriction V\) is a graph embedding from \(S\) to \(T\).
Since \(S\) is a combinatorial tree, it has no isolated vertices and therefore item \ref{item:b} of Lemma~\ref{Lemma : preserve edges} ensures that every edge \( (u,v)\) in \( S \) is preserved by \(f\).
For the converse, when \( u \) and \(v\) are not adjacent in \(S\), we have a sequence of vertices \(u=v_{0},\dotsc,v_{n}=v\) which is a path in \(S\), namely, such that \((v_{i},v_{i+1})\) is in \( S\), for every \(i<n\). Since \(f\) preserves edges and is one-to-one, the vertices \(f(v_{0}),\dotsc,f(v_{n})\) are all distinct and \((f(v_{i}),f(v_{i+1}))\) is an edge in \( T\), for every \(i<n\). As a result, we have that \(f(u)\) and \(f(v)\) are not adjacent in \(T\) by is acyclicity.
\end{proof}

The arguments of Fried and Koll\'ar show that for any $T$, the automorphisms of
$K_T$ are uniquely determined by their action on $V$, so we have the following.

\begin{corollary}\label{Corollary : auto}
The groups \(\Aut{K_{T}}\) and \(\Aut{T}\) are isomorphic via the map sending any automorphism \(\phi\) of \(K_{T}\) to the restriction of \(\phi\) to \(V\).
\end{corollary}

Now we use Theorem~\ref{Theorem : field complete} and Corollary~\ref{Corollary : auto} to prove that \(\embeds_{\mathrm{Fld},p}\) is invariantly universal.

\begin{theorem}
For \(p\) not equal to 2,
the quasi-order \( \embeds_{\mathrm{Fld},p} \) is invariantly universal.
\end{theorem}
\begin{proof}
It suffices to check that \(\embeds_{\mathrm{Fld},p} \) and
\( \iso_{\mathrm{Fld}} \) satisfy conditions \ref{condition : 1}--\ref{condition : 3} of Theorem~\ref{Theorem : CMMR13}. Let \(f\colon\XX\to X_{\mathrm{Fld},p}\) be the map sending
\(T\) to \( K_{T} \).
Theorem~\ref{Theorem : field complete} gives \ref{condition : 1}.
To see \ref{condition : 2}, notice that if \( \phi\colon K_{S}\to K_{T} \) is an isomorphism
then \(\phi\restriction V\) is an isomorphism from \(S\) to \(T\) as \((\phi \restriction V)^{-1}=\phi^{-1}\restriction V\).
Moreover, condition \ref{condition : 3} is immediate as the map \(T\mapsto \Aut{K_{T}} \) is the constant map
\(T\mapsto \set{id}\) by Corollary~\ref{Corollary : auto}. 
\end{proof}
\begin{corollary}
For every \( \analytic \) quasi-order \( P \) there is an \( \L_{\omega_1 \omega} \)-elementary class of countable fields of characteristic \(p\) such that the embeddability relation on it is Borel bi-reducible with \( P \).
\end{corollary}

\begin{question}
Is the embeddability relation \(\embeds_{\mathrm{Fld},2}\) between countable fields of characteristic \(2\) an invariantly universal quasi-order?
\end{question}

\subsection*{Acknowledgements}
The first author was supported during this research by 
EPSRC Early Career Fellowship EP/K035703/2, 
``Bringing set theory and algebraic topology together.''
This work was carried out while the second author was visiting Rutgers University supported by the ``National Group for the Algebraic and Geometric Structures and their Applications'' (GNSAGA--INDAM). The second author would like to thank Simon Thomas for interesting discussions and pointing out \cite{FriKol}.
We thank the anonymous referee for their careful reading and suggestions.

\end{document}